\documentclass[12pt]{amsart}

\usepackage{hyperref}
\usepackage{amssymb, amsthm, enumerate, mathptmx}

\usepackage{color}
\usepackage{cleveref}
\usepackage{comment}
\usepackage{varwidth}
\usepackage{extarrows}

\newtheorem{thm}{Theorem}[section]
\newtheorem{lem}[thm]{Lemma}
\newtheorem{cor}[thm]{Corollary}
\newtheorem{prop}[thm]{Proposition}

\newtheorem{conj}[thm]{Conjecture}
\newtheorem{prob}[thm]{Problem}

\theoremstyle{definition} 
\newtheorem{defn}[thm]{Definition} 
\newtheorem{example}[thm]{Example}

\DeclareMathOperator{\codim}{codim}
\DeclareMathOperator{\pd}{pd}

\DeclareMathOperator{\ch}{char}
\DeclareMathOperator{\Sym}{Sym}
\DeclareMathOperator{\Inc}{Inc}
\DeclareMathOperator{\Tor}{Tor}
\DeclareMathOperator{\Ext}{Ext}
\DeclareMathOperator{\reg}{reg}
\DeclareMathOperator{\ind}{ind}
\DeclareMathOperator{\ini}{in}
\DeclareMathOperator{\Ker}{Ker}
\DeclareMathOperator{\im}{Im}
\DeclareMathOperator{\Min}{Min}

\def\Icc{{\mathcal I}}

\newcommand{\N}{\mathbb{N}}
\newcommand{\Z}{\mathbb{Z}}

\newcommand{\ba}{{\bf a}}

\title[Asymptotic behavior of symmetric ideals]{Asymptotic behavior of symmetric ideals:\\ A brief survey}

\author{Martina Juhnke-Kubitzke}
\address{Institut f\"ur Mathematik, Universit\"at Osnabr\"uck, 49069 Osnabr\"uck, Germany}
\email{juhnke-kubitzke@uos.de}

\author{Dinh Van Le}
\address{Institut f\"ur Mathematik, Universit\"at Osnabr\"uck, 49069 Osnabr\"uck, Germany}
\email{dlevan@uos.de}

\author{Tim R\"{o}mer}
\address{Institut f\"ur Mathematik, Universit\"at Osnabr\"uck, 49069 Osnabr\"uck, Germany}
\email{troemer@uos.de}

\begin{document}


\begin{abstract}
Recently, chains of increasing symmetric ideals have attracted considerable attention. In this note, we summarize some results and open problems concerning the asymptotic behavior of several algebraic and homological invariants along such chains, including codimension, projective dimension, Castelnuovo-Mumford regularity, and Betti tables.
\end{abstract}


\maketitle

\section{Introduction}

This note provides a brief survey on recent developments in the study of asymptotic properties of chains of ideals in increasingly larger polynomial rings that are invariant under the action of symmetric groups.
Such chains have received much attention in the last decades as they arise naturally in various areas of mathematics, including
algebraic chemistry \cite{AH07, Dr10}, algebraic statistics \cite{BD11,Dr14,DEKL15,DK14,HM13,HS12,HS07,SS03}, 
group theory \cite{Co67}, and
representation theory \cite{CEF, GS18, NR17+, PS, SS-14,SS-16}.

As we will see, chains of ascending symmetric ideals are intimate related to their limits in a polynomial ring in infinitely many variables. Let us begin by fixing some notation. Throughout the paper, $\N$ denotes the set of positive integers, $c$ an element of $\N$, and $K$ an arbitrary field. Let
\[
 R=K[x_{k,j}\mid 1\le k\le c, j\ge 1]
\]
be the polynomial ring in ``$c\times\N$'' variables over $K$, and for each $n\in \N$ consider the following Noetherian subring of $R$:
\[
 R_n=K[x_{k,j}\mid 1\le k\le c,1\le j\le n].
\]
It is useful to view $R$ as the limit of the following ascending chain of polynomial rings 
\[
R_1\subseteq R_2\subseteq\cdots\subseteq R_n\subseteq\cdots.
\]
Let $\Sym(\infty):=\bigcup_{n\ge1}\Sym(n)$ denote the infinite symmetric group, where $\Sym(n)$ is the symmetric group on $\{1,\dots,n\}$ regarded as stabilizer of $n+1$ in $\Sym (n+1)$. One can also view $\Sym(\infty)$ as the limit of the following ascending chain of finite subgroups
\[
\Sym(1)\subseteq \Sym(2)\subseteq\cdots\subseteq \Sym(n)\subseteq\cdots.
\]
Consider the action of $\Sym(\infty)$ on $R$ by acting on the second index of the variables, i.e.,
\[
 \sigma \cdot x_{k,j}=x_{k,\sigma(j)} \quad\text{for }\ \sigma\in \Sym(\infty),\ 1\le k\le c,\ j\ge 1.
\]
Observe that this action restricts to an action of $\Sym(n)$ on $R_n$. 

We say that an ideal $I\subseteq R$ is \emph{$\Sym(\infty)$-invariant} (or \emph{$\Sym$-invariant} for short) if $\sigma(f)\in I$ for every $f\in I$ and $\sigma\in\Sym(\infty)$. In order to investigate such an ideal, a natural approach is to consider the truncated ideals $I_n=I\cap R_n$ for $n\ge 1$. Note that the sequence of truncated ideals $(I_n)_{n\ge1}$ forms an ascending chain, and furthermore, it is \emph{$\Sym$-invariant} in the sense that 
\[
 \Sym(m)(I_n)\subseteq I_m\quad\text{for all }\ m\ge n\ge 1.
\]
Conversely, when $(I_n)_{n\ge1}$ is a $\Sym$-invariant chain, its limit in $R$, i.e. the union $\bigcup_{n\ge1}I_n$, forms a $\Sym$-invariant ideal in $R$.

Although $R$ is not a Noetherian ring, a number of useful finiteness results have been established for this ring. For instance, it is known that $\Sym$-invariant ideals in $R$ satisfy the ascending chain condition, or in other words, $R$ is \emph{$\Sym$-Noetherian}. This celebrated result was first discovered by Cohen in his investigation of the variety of metabelian groups \cite{Co67, Co87}. The result was later rediscovered by Aschenbrenner and Hillar \cite{AH07} and Hillar and Sullivant \cite{HS12} with motivations from finiteness questions in chemistry and algebraic statistics. Generalizations of the $\Sym$-Noetherianity of $R$ were obtained by Nagel and R\"{o}mer \cite{NR17+} in the context of FI-modules.

Based on the $\Sym$-Noetherianity of $R$, Nagel and R\"{o}mer \cite{NR17} introduced Hilbert series for $\Sym$-invariant chains and showed that they are rational functions (see also \cite{KLS} for another proof using formal languages and \cite{GN} for some explicit results in a special case). As a consequence, they determined the asymptotic behaviors of the codimension and multiplicity along $\Sym$-invariant chains: the codimension grows eventually linearly, whereas the multiplicity grows eventually exponentially. This result leads to the following general problem (see \cite[Problem 1.1]{LNNR2}):

\begin{prob}
 Study the asymptotic behavior of invariants along $\Sym$-invariant chains of ideals.
\end{prob}

The aim of this note is to briefly summarize some recent results and open problems arising from the study of the previous problem. Apart from the aforementioned results of Nagel and R\"{o}mer, we will discuss further results on the asymptotic behaviors of the codimension \cite{LNNR2}, the Castelnuovo-Mumford regularity \cite{LNNR, Mu, Ra}, the projective dimension \cite{LNNR2, Mu}, and the Betti table \cite{Mu, NR17+} along $\Sym$-invariant chains of ideals. For the sake of simplicity, some results will not be stated in their most general form. Moreover, for the reader's convenience, a large number of examples will be provided.

The note is divided into seven sections.  In \Cref{sec2} we recall some basic notions and facts on invariant chains of ideals. \Cref{sec3} contains Nagel-R\"{o}mer's result on rationality of Hilbert series and its consequences on the asymptotic behaviors of codimension and multiplicity, together with an improvement on the codimension obtained in \cite{LNNR2}. The asymptotic behaviors of the Castelnuovo-Mumford regularity, the projective dimension, and the Betti table are discussed in Sections \ref{sec4}, \ref{sec5}, and \ref{sec6}, respectively. Finally, some open problems are proposed in \Cref{sec7}.


\section{Preliminaries}\label{sec2}

We use the notation and definitions from the introduction. In particular, $c$ is a fixed positive integer, $R$ is the polynomial ring $K[x_{k,j}\mid 1\le k\le c, j\ge 1]$, and for each $n\ge 1$, $R_n$ is the subring of $R$ generated by the first $c\times n$ variables. Moreover, for any monomial order $\le$ on $R$, we will use the same notation to denote its restriction to $R_n$.

Let $I\subseteq R$ be a $\Sym$-invariant ideal and consider the chain $(I_n)_{n\ge 1}$ of truncations of $I$. A useful technique for investigating the asymptotic properties of the chain $(I_n)_{n\ge 1}$ is to pass to the chain $(\ini_{\le}(I_n))_{n\ge 1}$ of initial ideals. But one issue then arises: unfortunately, the chain $(\ini_{\le}(I_n))_{n\ge 1}$ is typically not $\Sym$-invariant (see \Cref{ex1}). The reason is that the action of the symmetric group $\Sym(\infty)$ on $R$ is not compatible with monomial orders: for any monomial order $\leq$ on $R$, there exist monomials $u,v\in R$ with $u<v$ and some $\sigma\in \Sym (\infty)$ such that $\sigma(u)>\sigma(v)$; see \cite[Remark 2.1]{BD11}. To deal with this issue, one introduces another action on $R$ with a larger class of invariant ideals which behaves better with respect to monomial orders.

Consider the following monoid of strictly increasing functions on $\N$:
\[
 \Inc = \{\pi \colon \N \to \N \mid \pi(j)<\pi(j+1)\ \text{ for all }\ j\ge 1\}.
\]
Analogously to $\Sym(\infty)$, one defines the action of $\Inc$ on $R$ as follows\footnote{Let $i\ge0$ be an integer and consider the following submonoid of $\Inc$ that fixes the first $i$ natural numbers:
 \[
\Inc^i=\{\pi \in \Inc \mid \pi(j)=j\ \text{ for all }\ j\le i\}.
\]
Thus, in particular, $\Inc=\Inc^0$. We note that the notions and results presented in this survey can be extended to a more general setting, where the monoid $\Inc$ is replaced by $\Inc^i$ for any $i\ge0$. However, for the sake of simplicity, we will only concentrate on the monoid $\Inc$, referring the reader to \cite{LNNR,LNNR2,NR17,NR17+} for the general case.}:
\[
 \pi \cdot x_{k,j}=x_{k,\pi(j)} \quad\text{for }\ \pi\in \Inc,\ 1\le k\le c,\ j\ge 1.
\]

\begin{defn}
\label{def1}
 An ideal $I\subseteq R$ is called \emph{$\Inc$-invariant} if 
 \[
  \pi(I):=\{\pi(f)\mid f\in I\}\subseteq I\quad \text{for all }\ \pi\in \Inc.
 \]
 A chain of ideals $(I_n)_{n\ge 1}$ with $I_n\subseteq R_n$ is \emph{$\Inc$-invariant} if
\[
 \Inc_{m,n}(I_m):=\{\pi(I_m)\mid \pi\in \Inc_{m,n}\}\subseteq I_n\quad \text{for all }\ m\le n,
\]
where
\[
\Inc_{m,n}:=\{\pi \in \Inc \mid \pi(m)\le n\}.
\]
\end{defn}

It is evident that if $I\subseteq R$ is an {$\Inc$-invariant} ideal, then its truncations $I\cap R_n$ form an $\Inc$-invariant chain. Conversely, if $(I_n)_{n\ge 1}$ is an $\Inc$-invariant chain, then $I:=\bigcup_{n\ge 1}I_n$ is an {$\Inc$-invariant} ideal in $R$. 

Although $\Inc$ is not a submonoid of $\Sym(\infty)$, it turns out that the class of $\Inc$-invariant ideals contains the one of $\Sym$-invariant ideals (see, e.g., \cite[Lemma 7.6]{NR17}):

\begin{prop}
\label{prop22}
 For any $f\in R_m$ and $\pi\in\Inc_{m,n}$ with $m\le n$, there exists $\sigma\in\Sym(n)$ such that $\pi(f)=\sigma(f)$. Thus, it holds that $\Inc_{m,n}(f)\subseteq\Sym(n)(f)$. In particular:
 \begin{enumerate}
  \item 
  every $\Sym$-invariant ideal $I\subseteq R$ is also an $\Inc$-invariant ideal;
  \item
  every $\Sym$-invariant chain $(I_n)_{n\ge 1}$ is also an $\Inc$-invariant chain.
 \end{enumerate} 
\end{prop}

In addition, the class of $\Inc$-invariant ideals is closed under the action of certain monomial orders on $R$, which is not true for the class of $\Sym$-invariant ideals (see \Cref{lem_initial_filtration} and \Cref{ex1} below). We say that a monomial order $\le$ \emph{respects} $\Inc$ if $\pi(u)\le \pi(v)$ whenever $\pi \in \Inc$ and $u,v$ are monomials of $R$ with $u\le v$. This condition implies that
\[
\ini_{\le}(\pi(f))=\pi(\ini_{\le}(f))\quad \text{for all } f\in R\ \text{ and } \pi \in \Inc.
\]
Examples of  monomial orders respecting $\Inc$ include the lexicographic order and the reverse-lexicographic order on $R$ induced by the following ordering of the variables:
\begin{equation}
\label{order}
 x_{k,j}\le x_{k',j'}\quad \text{if either $k<k'$ or $k=k'$ and $j<j'$}.
\end{equation}
The next simple lemma makes it possible to use the Gr\"{o}bner bases method in studying $\Inc$-invariant ideals, and hence, $\Sym$-invariant ideals as well:

\begin{lem}
	\label{lem_initial_filtration}
	Let $(I_n)_{n\ge 1}$ be an $\Inc$-invariant chain of ideals. Then for any monomial order $\le$ respecting $\Inc$, the chain $(\ini_{\le}(I_n))_{n\ge 1}$ is also $\Inc$-invariant. In particular, if $I\subseteq R$ is  an $\Inc$-invariant ideal, then $\ini_{\le}(I):=\bigcup_{n\ge 1}\ini_{\le}(I\cap R_n)$ is also an $\Inc$-invariant ideal in $R$.
\end{lem}

\begin{proof}
 For the first assertion see, e.g., \cite[Lemma 7.1]{NR17}. The second assertion follows from the first one and the observation after \Cref{def1}.
\end{proof}


\begin{example}[{see also \cite[Example 2.2]{LNNR2}}]
	\label{ex1}
	Let $R=K[x_j\mid j\in\N]$ (i.e., $c=1$) with the field $K$ having characteristic $0$. Consider the $\Sym$-invariant ideal $I$ generated by $x_1^2+x_2x_3$. Let $(I_n)_{n\ge 1}$ be the chain of truncations of $I$. Then $I_1=I_2=\langle0\rangle$ and
	\begin{align*}
		I_3&=\Sym(3)(x_1^2+x_2x_3)=\langle x_1^2+x_2x_3,\; x_2^2+x_1x_3,\; x_3^2+x_1x_2\rangle,\\
		I_4&=\Sym(4)(x_1^2+x_2x_3)\\
		&=I_3+\langle x_1^2+x_2x_4,\; x_1^2+x_3x_4,\; x_2^2+x_1x_4,\; x_2^2+x_3x_4,\\
		&\hspace{1.43cm}x_3^2+x_1x_4,\; x_3^2+x_2x_4,\;
		x_4^2+x_1x_2,\; x_4^2+x_1x_3,\; x_4^2+x_2x_3\rangle.
	\end{align*}
	Computations with Macaulay2 \cite{GS} using the reverse-lexicographic order on $R$ induced by $x_1<x_2<x_3<\cdots$ give
	\begin{align*}
		\ini(I_3)&=\langle x_2^2, \; x_3x_2,\; x_3^2,\; x_2x_1^2,\; x_3x_1^2,\; x_1^4 \rangle,\\
		\ini(I_4)&=\langle x_2x_1,\; x_3x_1,\; x_4x_1,\; x_2^2,\; x_3x_2,\; x_4x_2,\; x_3^2,\; x_4x_3,\; x_4^2,\; x_1^3 \rangle.
	\end{align*}
	Since $x_2^2\in \ini(I_3)$ but $x_1^2\not\in \ini(I_3),$ we see that $\Sym(3)(\ini(I_3))\nsubseteq\ini(I_3)$. 
	Thus, the chain $(\ini(I_n))_{n\ge 1}$ is not $\Sym$-invariant. Moreover, the ideal $\ini(I)=\bigcup_{n\ge 1}\ini(I_n)$ is also not $\Sym$-invariant: otherwise, $x_1^2$ would belong to $\ini(I)$, which implies $x_1^2\in I$ (since $x_1^2$ is the smallest element of $R$ of degree 2), contradicting the fact that $I_1=0$.
	
	On the other hand, the chain $(\ini(I_n))_{n\ge 1}$ and the ideal $\ini(I)$ are both $\Inc$-invariant by \Cref{prop22} and \Cref{lem_initial_filtration}. For instance, one can check that
	\[
	\Inc_{3,4}(\ini(I_3))=\ini(I_3)+\langle x_4x_2,\; x_4x_3,\; x_4^2,\; x_4x_1^2 \rangle\subseteq\ini(I_4).
	\]
\end{example}

\Cref{prop22}, \Cref{lem_initial_filtration}, and \Cref{ex1} suggest that even if one is primarily interested in $\Sym$-invariant chains of ideals it is worthwhile and often more convenient to study the larger class of $\Inc$-invariant chains.

To conclude this section let us recall the following fundamental finiteness result (see \cite[Theorems 3.1, 3.6, Corollaries 3.5, 3.7]{HS12} and also \cite[Corollaries 3.6, 5.4, Lemma 5.2]{NR17}):

\begin{thm}
\label{Noetherianity}
 The following statements hold:
 \begin{enumerate}
  \item 
  The ring $R$ is $\Inc$-Noetherian, i.e., for any $\Inc$-invariant ideal $I\subseteq R$ there exist finitely many elements $f_1,\dots,f_m\in R$ such that
  \[
   I=\langle\Inc(f_1),\dots,\Inc(f_m)\rangle.
  \]
  \item
  Every $\Inc$-invariant chain $\Icc=(I_n)_{n\ge 1}$ stabilizes, i.e., there exists an integer $r\ge1$ such that for all $n\ge m\ge r$ one has
\[
I_n=\langle\Inc_{m,n}(I_m)\rangle
\]
as ideals in $R_n$. The least integer $r$ with this property is called the {\rm stability index} of $\Icc$, denoted by $\ind(\Icc).$
 \end{enumerate}
 
In particular, analogous statements hold if $\Inc$ is replaced by $\Sym(\infty)$.
\end{thm}

\begin{example}
 Let $R=K[x_j\mid j\in\N]$ with $\ch K=0$ and $n_0$ a natural number. Consider the ideal $I\subseteq R$ generated by the $\Sym$-orbits of the elements $x_1^n+\cdots+x_n^n$ for all $n\ge n_0$, i.e.,
 \[
  I=\langle\Sym(\infty)(x_1^n+\cdots+x_n^n)\mid n\ge n_0\rangle.
 \]
 It is evident that $I$ is a $\Sym$-invariant ideal. So by the previous theorem, $I$ has a finite set of generators up to the action of $\Sym(\infty)$. How many elements do we need to generate $I$?\footnote{We thank H. Brenner for asking this question and Hop D. Nguyen for suggesting the answer.} Quite surprisingly, one polynomial is enough, and in fact, $I$ is a monomial ideal. We claim that
 \[
  I=\langle\Sym(\infty)(x_1^{n_0})\rangle.
 \]
 Let $J$ denote the ideal on the right hand side. Then clearly $I\subseteq J$. To show the reverse inclusion observe that for $j>1$ one has
 \[
  \begin{aligned}
   x_j^{n_0}-x_1^{n_0}& =(x_{j}^{n_0}+x_{j+1}^{n_0}+\cdots+x_{j+n_0-1}^{n_0})-(x_{1}^{n_0}+x_{j+1}^{n_0}+\cdots+x_{j+n_0-1}^{n_0})\\
   &\in\langle\Sym(\infty)(x_{1}^{n_0}+\cdots+x_{n_0}^{n_0})\rangle\subseteq I.
  \end{aligned}
 \]
 It follows that
 \[
  n_0x_1^{n_0}=(x_{1}^{n_0}+\cdots+x_{n_0}^{n_0})-(x_2^{n_0}-x_1^{n_0})-\cdots-(x_{n_0}^{n_0}-x_1^{n_0})\in I.
 \]
 This gives $x_1^{n_0}\in I$ since $\ch K=0$. Hence, $I=J$.
\end{example}


\section{Hilbert series, multiplicity, and codimension}\label{sec3}

The celebrated Hilbert theorem says that the Hilbert series of any finitely generated graded module over a standard graded $K$-algebra $S$ is a rational function, and furthermore, this rational function encodes the dimension as well as the multiplicity of the module. In particular, given a proper graded ideal $J\subset S$, the Hilbert series of the quotient ring $S/J$,
\[
 H_{S/J}(t)= \sum_{u\ge0}\dim_K(S/J)_ut^u,
\]
can be uniquely expressed in the form
\[
 H_{S/J}(t)=\frac{Q(t)}{(1-t)^d}
\]
with $Q(t)\in\Z[t]$ and $Q(1)>0$. Moreover, one has that $d=\dim (S/J)$ and $Q(1)=e(S/J)$, the multiplicity of $S/J$. Note that $e(S/J)$ is also called the degree of $J$, denoted by $\deg (J)$.

The above result was extended to $\Inc$-invariant chains of ideals by Nagel and R\"{o}mer \cite{NR17}. They defined the Hilbert series for such a chain, showed its rationality, and obtained from that the asymptotic behaviors of the codimension and the degree of the ideals in the chain. Let us now summarize their results.

Let $\Icc=(I_n)_{n\ge 1}$ be an $\Inc$-invariant chain of graded ideals. Thus, each $I_n$ is a graded ideal in the polynomial ring $R_n=K[x_{k,j}\mid 1\le k\le c,1\le j\le n]$. The \emph{(equivariant) Hilbert series} of $\Icc$ is defined as the following bivariate formal power series
\[
 H_{\Icc}(s,t)= \sum_{n\ge0}H_{R_n/I_n}(t)s^n= \sum_{n\ge0,\;u\ge0}\dim_K(R_n/I_n)_us^nt^u.
\]
This series was introduced in \cite{NR17} as a way to encode all the Hilbert series of the quotient rings $R_n/I_n$ simultaneously. See also \cite{SS-14, Sn} for related notions. The following crucial fact is given in \cite[Proposition 7.2]{NR17}.

\begin{thm}
\label{Hilbert}
 Let $\Icc=(I_n)_{n\ge 1}$ be an $\Inc$-invariant chain of graded ideals. Then the Hilbert series of $\Icc$ is a rational function of the form
 \[
  H_{\Icc}(s,t)=\frac{g(s,t)}{(1-t)^a\prod_{l=1}^b[(1-t)^{c_l}-s\cdot f_{l}(t)]},
 \]
where $a,\;b,\;c_l\in\Z_{\ge0}$ with $c_l\le c$, $g(s,t)\in\Z[s,t]$, $f_{l}(t)\in\Z[t]$ and $f_{l}(1)>0$ for every $l=1,\dots,b$.
\end{thm}

The proof of this result is rather long and involved. Nevertheless, the proof technique is very useful, as it has been employed in \cite{LNNR, LNNR2} to study the asymptotic behavior of the Castelnuovo-Mumford regularity, codimension and projective dimension along $\Inc$-invariant chains, the main results of which are summarized below. A shorter proof of the rationality of $H_{\Icc}(s,t)$ using formal languages is given by Krone, Leykin and Snowden \cite{KLS}. This approach, however, does not seem to yield information on the denominator of $H_{\Icc}(s,t)$ as stated in \Cref{Hilbert}. In the case that the chain $\Icc$ is generated by one monomial, the rational form of $H_{\Icc}(s,t)$ can be determined explicitly, as carried out in \cite{GN}. Recently, inspired by the independent set theorem \cite[Theorem 4.7]{HS12}, Maraj and Nagel \cite{MN} define multigraded Hilbert series of chains of ideals that are invariant under the action of a product of symmetric groups. They found a sufficient condition for the rationality of such series; see \cite[Theorem 3.5]{MN}.

\begin{example}
\label{ex32}
 Let $c=2$ and consider the ideal
 \[
  I=\langle x_{1,j}x_{2,k}\mid j,k\in\N\rangle\subseteq R.
 \]
It is clear that $I$ is an $\Inc$-invariant ideal. Let $\Icc=(I_n)_{n\ge 1}$ denote the chain of truncations of $I$. Thus,
\[
 I_n=I\cap R_n=\langle x_{1,j}x_{2,k}\mid 1\le j,k\le n\rangle.
\]
One may view $I_n$ as the edge ideal of the complete bipartite graph $K_{n,n}$ with vertex set $\{x_{1,1},\dots,x_{1,n}\}\cup \{x_{2,1},\dots,x_{2,n}\}$. So it is well-known that (see, e.g., \cite[Exercise 7.6.11]{Vi} or \cite[Theorem 2.1]{CN})
\[
 H_{R_n/I_n}(t)=\frac{2}{(1-t)^n}-1.
\]
Hence, we get
\[
 \begin{aligned}
  H_{\Icc}(s,t)&=1+\sum_{n\ge1}\Big(\frac{2}{(1-t)^n}-1\Big)s^n=\sum_{n\ge0}\frac{2s^n}{(1-t)^n}-\sum_{n\ge0}s^n\\
  &=\frac{2(1-t)}{1-t-s}-\frac{1}{1-s}=\frac{2(1-s)(1-t)-1}{(1-t-s)(1-s)}.
 \end{aligned}
\]
\end{example}

\Cref{Hilbert} has implications for the asymptotic behaviors of the codimension and degree of ideals in $\Inc$-invariant chains, which might be regarded as analogous to the fact mentioned above that the Hilbert series of an ideal encodes the codimension and degree of the ideal. The next result follows from \cite[Theorem 7.10]{NR17}.

\begin{cor}
\label{codim+degree}
 Let $\Icc=(I_n)_{n\ge 1}$ be an $\Inc$-invariant chain of proper graded ideals. Then the following statements hold:
 \begin{enumerate}
  \item 
  The codimension of $I_n$ is eventually a linear function on $n$. More precisely, there are integers $A$, $B$ with $0\le A\le c$ such that
  \[
   \codim (I_n)=An+B\quad\text{for all }\ n\gg0.
  \]
  \item
  The degree of $I_n$ grows eventually exponentially. More precisely, there are integers $M>0$, $L\ge0$ and a rational number $Q>0$ such that
  \[
   \lim_{n\to\infty}\frac{\deg (I_n)}{M^nn^L}=Q.
  \]
 \end{enumerate}

\end{cor}

The linear function $\codim (I_n)$ (for $n\gg0$) is further investigated in \cite{LNNR2}, where its leading coefficient is (somewhat combinatorially) determined. We will discuss this result in the remaining part of this section.

In view of \Cref{lem_initial_filtration}, we may assume that $\Icc=(I_n)_{n\ge 1}$ is an $\Inc$-invariant chain of monomial ideals. 
Let $G(I_n)$ denote the minimal set of monomial generators of $I_n$. The following notion is essential for determining the growth of the function $\codim (I_n)$.

\begin{defn}
	\label{def:i cover}
	Let $C$ be a subset of $[c]=\{1,\dots,c\}$. Assume that $I_n$ is a proper monomial ideal in $R_n$. We say that $C$ is a \emph{cover} of $I_n$ if for every $u\in G(I_n)$ there exist $k\in C$ and $j\ge 1$ such that $x_{k,j}$ divides $u$.
Set
\[
\gamma(I_n)=\min\{\#C \ | \ C \text{ is a cover of $I_n$} \}.
\] 
\end{defn}

\begin{example}
 Let $\Icc=(I_n)_{n\ge 1}$ be the $\Inc$-invariant chain in \Cref{ex32}. It is clear that each $I_n$ has 2 minimal covers, namely, $C_1=\{1\}$ and $C_2=\{2\}$. Thus,
 \[
 \gamma(I_n)=1\quad\text{for all }\ n\ge 1.
\]
\end{example}

There is an alternative way to compute $\gamma(I_n)$ using minimal primes of $I_n$. Let $\varphi$ be the map defined on the variables of $R$ by
\[
\varphi(x_{k,j})=k \quad\text{for every }\ j\ge 1.
\]
Then $\varphi$ clearly induces a map, still denoted by $\varphi$, from the set $\Min(I_n)$ of minimal primes of $I_n$ to the set of covers of $I_n$. Moreover, it can be shown that any minimal cover of $I_n$ is of the form $\varphi(P)$ for some $P\in \Min(I_n)$; see \cite[Proposition 3.3]{LNNR2}. Therefore,
\[
 \gamma(I_n)=\min\{\#\varphi(P)\mid P\in \Min(I_n)\}.
\]

\begin{example}
Assume $c\ge 4$ and consider the ideal
\[
I_3=\langle x_{1,2}^2x_{2,3}^3,\; x_{1,1}x_{3,2}^2,\; x_{4,2}^2\rangle \subset R_3.
\]
We have that
\[
  \Min(I_3)=\{\langle x_{1,1},\; x_{1,2},\;x_{4,2}\rangle,\ \langle x_{1,1},\;x_{2,3},\; x_{4,2}\rangle,\ \langle x_{1,2},\; x_{3,2},\; x_{4,2}\rangle,\ \langle x_{2,3},\; x_{3,2},\; x_{4,2}\rangle\}.
\]
Thus,
\[
 \varphi(\Min(I_3))=\{\{1,4\},\ \{1,2,4\},\ \{1,3,4\},\ \{2,3,4\}\},
\]
and so
\[
\gamma(I_3)=\min\{\#\varphi(P)\mid P\in \Min(I_3)\}=2.
\]
One can also compute $\gamma(I_3)$ by observing that $I_3$ has two minimal covers: $C_1=\{1,4\}$ and $C_2=\{2,3,4\}$.
\end{example}

Given an $\Inc$-invariant chain of proper monomial ideals $\Icc=(I_n)_{n\ge 1}$, it follows readily from \Cref{def:i cover} that $\gamma$ is a non-decreasing function along this chain, in the sense that $\gamma(I_n)\le \gamma(I_{n+1})$ for all $n\ge 1$ (see \cite[Lemma 3.5]{LNNR2}). Since $0\le \gamma(I_n)\le c$ by definition, $\gamma$ must stabilize, i.e., $\gamma(I_n)= \gamma(I_{n+1})$ for all $n\gg0$. In fact, one has $\gamma(I_n)= \gamma(I_{r})$ for all $n\ge r$, where $r=\ind(\Icc)$ denotes the stability index of $\Icc$. This is because the set of covers of $I_n$ is stable for $n\ge r$, which follows from the fact that the action of $\Inc$ on $R$ keeps the first index of the variables unchanged (see \cite[Lemma 3.6]{LNNR2}). So we may define
\[
 \gamma(\Icc)= \gamma(I_n)\quad\text{for some }\  n\ge \ind(\Icc). 
\]
This number is exactly the leading coefficient of the function $\codim (I_n)$ when $n\gg0$ (see \cite[Theorem 3.8]{LNNR2}):

\begin{thm}
 \label{codim}
 Let $\Icc=(I_n)_{n\ge 1}$ be an $\Inc$-invariant chain of proper monomial ideals. Then there exists an integer $B$ such that
 \[
  \codim (I_n)=\gamma(\Icc)n + B \quad \text{for all }\ n\gg 0.
 \]
\end{thm}

As a consequence, we obtain the following more explicit and slightly more general version of \Cref{codim+degree}(i) (see \cite[Corollary 3.12]{LNNR2}). Note that this result is applicable also to non-graded ideals. 

\begin{cor}
 \label{dim}
 Let $\Icc=(I_n)_{n\ge 1}$ be an $\Inc$-invariant chain of proper ideals, and let $\le$ be any monomial order on $R$ respecting $\Inc$. 
 Then there exists an integer $B$ such that 
  \[
  \codim (I_n)=\gamma(\ini_{\le}(\Icc))n + B\quad \text{for all }\ n\gg 0. 
 \]
 In particular, the coefficient $\gamma(\ini_{\le}(\Icc))$ does not depend on the order $\le$ and therefore will be simply denoted by $\gamma(\Icc)$.
\end{cor}

\begin{example}
\label{ex39}
 Let $p$ be a positive integer with $p\le c$. For $n\ge 1$ we view the variables of $R_n$ as a matrix $X_{c\times n}$ of size $c\times n$ and we denote by $I_n$ the ideal generated by all $p$-minors of this matrix. Then the chain $\Icc=(I_n)_{n\ge 1}$ is evidently $\Sym$-invariant, and therefore, it is also $\Inc$-invariant. Let $\le$ be a \emph{diagonal term order} on $R$ that respects $\Inc$ (e.g., the lexicographic order extending the order of the variables as given in \eqref{order}). Then by \cite[Theorem 1]{St} (see also \cite[Theorem 5.3]{BC}), the $p$-minors of $X_{c\times n}$ form a Gr\"{o}bner basis for $I_n$ with respect to $\le$ for all $n\ge p$. It follows that
 \[
 \ini_{\le}(I_n)=\langle x_{k_1,j_1}\cdots x_{k_p,j_p}\mid 1\le k_1<\cdots <k_p\le c,\ 1\le j_1<\cdots <j_p\le n \rangle\ \text{ for }\ n\ge p.
 \]
 Now one can easily check that the minimal covers of  $\ini_{\le}(I_n)$ for $n\ge p$ are exactly subsets of $[c]$ of cardinality $c-p+1$. Thus, $\gamma(\ini_{\le}(\Icc))=c-p+1$, and so \Cref{dim}  gives
 \[
 \codim (I_n)=(c-p+1)n + B\quad \text{for all }\ n\gg 0,
 \]
 where $B$ is a constant integer. Note that an exact formula for $\codim (I_n)$ is known, namely,
 \[
 \codim (I_n)=(c-p+1)(n-p+1)\quad \text{for all }\ n\ge p
 \]
 (see, e.g., \cite[Theorem 6.8]{BC}).
\end{example}

\Cref{codim} and \Cref{dim} provide a convenient way to compute the constant $A$ in \Cref{codim+degree}. It is therefore desirable to ask for similar results for the other constants.

\begin{prob}
 Determine the constants $B$, $M$, $L$, $Q$ in \Cref{codim+degree}.
\end{prob}


In the following we will concentrate on some homological invariants, such as Betti numbers, the Castelnuovo-Mumford regularity or the projective dimension. Let us briefly recall these notions. 

Let $S=K[y_1,\dots,y_m]$ be a standard graded polynomial ring and $J\subsetneq S$ a nonzero graded ideal. Assume that the minimal graded free resolution of $J$ is given by
\[
 0\rightarrow \bigoplus_{j\in\Z}S(-j)^{\beta_{p,j}(J)}\xrightarrow{\partial_p}\cdots\xrightarrow{\partial_2} \bigoplus_{j\in\Z}S(-j)^{\beta_{1,j}(J)}\xrightarrow{\partial_1} \bigoplus_{j\in\Z}S(-j)^{\beta_{0,j}(J)}\xrightarrow{\partial_0} J\xrightarrow{\partial_{-1}} 0,
\]
where $\partial_{-1}$ is the zero map. Then the numbers $\beta_{i,j}(J)$ are called the \emph{graded Betti numbers} of $J$. They can be expressed as
\[
 \beta_{i,j}(J)=\dim_K\Tor^S_i(J,K)_j.
\]
The graded Betti numbers of $J$ are usually given in a table, called the \emph{Betti table} of $J$, in which the entry in the $i$-th column and $j$-th row is $\beta_{i,i+j}(J)$; see \Cref{fig1}.

\begin{figure}[h]
\begin{tabular}[t]{c|cccc}
$J$&0&1&2&\dots\\
\hline 
\text{0}&$\beta_{0,0}(J)$&$\beta_{1,1}(J)$&$\beta_{2,2}(J)$&\dots\\
\text{1}&$\beta_{0,1}(J)$&$\beta_{1,2}(J)$&$\beta_{2,3}(J)$&\dots\\
\text{2}&$\beta_{0,2}(J)$&$\beta_{1,3}(J)$&$\beta_{2,4}(J)$&\dots\\
\text{\vdots}&\vdots&\vdots&\vdots&
\end{tabular}
\caption{Betti table of $J$}
\label{fig1}
\end{figure}

For $i\ge 0$ the module
\[
 \im(\partial_i)=\Ker(\partial_{i-1})
\]
is called the \emph{$i$-th syzygy module} of $J$. The \emph{projective dimension} and the \emph{Castelnuovo-Mumford regularity} (or \emph{regularity} for short) of $J$ are defined as 
\begin{align*}
\pd(J)&=\max\{i\mid \beta_{i,j}(J)\ne0\ \text{ for some }\ j\},\\
\reg(J)&=\max\{j\mid\beta_{i,i+j}(J)\ne0\ \text{ for some }\ i\}.
\end{align*}
Thus, $\pd(J)$ and $\reg(J)$, respectively, are the index of the last nonzero column and last nonzero row of the Betti table of $J$. Note that $\pd(J)$ and $\reg(J)$ can also be interpreted as 
\begin{align*}
\pd(J)&=\max\{i\mid \Ext_S^i(J,S)_j\ne0\ \text{ for some }\ j\},\\
\reg(J)&=\max\{j\mid\Ext_S^i(J,S)_{-i-j}\ne0\ \text{ for some }\ i\}.
\end{align*}
Moreover, by Hilbert's Syzygy Theorem (see, e.g., \cite[Theorem 15.2]{Pe11}) one always has
\[
 \pd(J)\le \dim (S)-1=m-1.
\]

\section{Castelnuovo-Mumford regularity}\label{sec4}

The regularity of powers of a graded ideal in a polynomial ring is eventually a linear function. This beautiful result was independently proven by Cutkosky-Herzog-Trung \cite{CHT} and Kodiyalam \cite{K}. A similar behavior is expected for the regularity of ideals in an $\Inc$-invariant chain, as proposed in \cite[Conjecture 1.1]{LNNR}:

\begin{conj}
	\label{conj}
	Let $\Icc=(I_n)_{n\ge 1}$ be a nonzero $\Inc$-invariant chain of graded ideals. Then $\reg(I_n)$ is eventually a linear function, that is, there exist integers $C$ and $D$ such that
	\[
	\reg(I_n)=Cn+D \quad \text{ whenever }\ n\gg 0.
	\]
\end{conj}

As evidence for this conjecture, a rather sharp upper linear bound for $\reg(I_n)$ is obtained in \cite{LNNR}. Additionally, some special cases of the conjecture are verified in \cite{LNNR,Mu,Ra}. This section is devoted to discussing these results.

Let us begin with an upper linear bound for $\reg(I_n)$. We first assume that $\Icc=(I_n)_{n\ge 1}$ is a nonzero $\Inc$-invariant chain of monomial ideals. As before, let $G(I_n)$ denote the minimal set of monomial generators of $I_n$. In order to bound $\reg(I_n)$, the following weights are introduced in \cite[Definition 3.2]{LNNR}:

\begin{defn}
	\label{def:i weight}
	Let $k\in [c]$. For a nonzero monomial $u\in R_n$ set
	\[\begin{aligned}
	w_k(u)&=
	\max\{e\mid \text{$x_{k,j}^e$ divides $u$ for some $j \ge 1$}\},\\
	w(u)&=\max\{w_k(u)\mid k\in [c]\}.
	\end{aligned}	
	\]
	Define the following weights for $I_n$:
		\[
		\begin{aligned}
		w_k(I_n) &= \max \{ w_k(u) \mid u \in G(I_n)\},\\
		\omega(I_n) &=
		\min \{ w(u) \mid u \in G(I_n)\}.
		\end{aligned}
		\]
\end{defn}

\begin{example}
\label{ex43}
Let $c= 3$ and consider the ideal $I_4=\langle u_1,\; u_2,\; u_3,\; u_4\rangle \subset R_4$ with
\[
u_1= x_{1,1}^2x_{2,1}^3x_{2,2},\ \ u_2=x_{1,3}^3x_{2,2}^4x_{3,2}^5,\ \ u_3=x_{3,1},\ \ u_4=x_{1,4}^2.
\]
Then we have that
\[
\begin{aligned}
  w_1(u_1)&=2,\ \ w_2(u_1)=3,\ \ w_3(u_1)=0, \ \ w(u_1)=3,\\
  w_1(u_2)&=3,\ \ w_2(u_2)=4,\ \ w_3(u_2)=5, \ \ w(u_2)=5,\\
  w_1(u_3)&=0,\ \ w_2(u_3)=0,\ \ w_3(u_3)=1, \ \ w(u_3)=1,\\
  w_1(u_4)&=2,\ \ w_2(u_4)=0,\ \ w_3(u_4)=0, \ \ w(u_4)=2.
\end{aligned}
\]
Thus,
\[
 w_1(I_4)=3,\ \ w_2(I_4)=4,\ \ w_3(I_4)=5, \ \ \omega(I_4)=1.
\]
\end{example}

Similarly to the function $\gamma$ in the previous section, the weights defined in \Cref{def:i weight} have the following stabilization property:
\[
 w_k(I_{n+1})= w_k(I_{n})\quad\text{and}\quad
	\omega(I_{n+1})= \omega(I_{n})\quad \text{for all }\ n\gg0.
\]
Indeed, let $r=\ind(\Icc)$ be the stability index of $\Icc$. Then for $n\ge r$ one has that
\[
 \langle G(I_{n+1})\rangle=I_{n+1}=\langle \Inc_{n,n+1}(I_n)\rangle=\langle \Inc_{n,n+1}(G(I_n))\rangle,
\]
giving
\[
 G(I_{n+1})\subseteq \Inc_{n,n+1}(G(I_n)).
\]
Using the latter inclusion one can show that (see \cite[Lemma 3.5, Remark 3.6]{LNNR})
\[
 w_k(I_{n+1})\le w_k(I_{n})\quad\text{and}\quad
	\omega(I_{n+1})= \omega(I_{n})
\]
for every $k\in [c]$ and $n\ge r$. Now since $w_k(I_{n})$ is a nonnegative integer, it must hold that $w_k(I_{n+1})= w_k(I_{n})$ for $n\gg0$. 

The stabilization of the weights $w_k$ and $\omega$ allows us to extend them to the chain $\Icc$ by setting
\[
\begin{aligned}
w_k(\Icc)&=w_k(I_{n})\quad \text{for }\ n\gg 0,\\
\omega(\Icc)&= \omega(I_{n})\quad \text{for }\ n\ge r.
\end{aligned}
\]

As the next example demonstrates, it can happen that $w_k(\Icc)\ne w_k(I_{r})$.
\begin{example}
 Consider again the ideal $I_4$ in \Cref{ex43} and let $\Icc=(I_n)_{n\ge 1}$ be an $\Inc$-invariant chain with $I_n=\langle\Inc_{4,n}(I_4)\rangle$ for all $n\ge 4$. Then $\ind(\Icc)=4$. Using induction on $n$ one can easily show that
 \[
  I_n=\langle x_{1,j}^2x_{2,j}^3 x_{2,k}\mid 1\le j<k\le n-2,\; j<4\rangle+\langle x_{3,j}\mid 1\le j\le n-3\rangle+\langle x_{1,j}^2\mid 4\le j\le n\rangle 
 \]
for all $n\ge 5.$ This yields
\[
 w_1(\Icc)=w_1(I_{n})=2, \  w_2(\Icc)=w_2(I_{n})=3, \  w_3(\Icc)=w_3(I_{n})=1, \ \omega(\Icc)=\omega(I_{n})=1
\]
for all $n\ge 5.$ So from \Cref{ex43} we see that $w_k(\Icc)\ne w_k(I_{4})$ for $k=1,2,3$.
\end{example}

Now we are ready to state an upper linear bound for $\reg(I_n)$ (see \cite[Theorem 4.1]{LNNR}):

\begin{thm}
	\label{thm_bounding_reg_monomial}
	Let $\Icc=(I_n)_{n\ge 1}$ be a nonzero $\Inc$-invariant chain of monomial ideals. Set 
	\[
	C(\Icc)=\max\{\omega(\Icc)-1,0\}+\max\Big\{\sum_{k\ne l}w_k(\Icc)\mid l\in[c]\Big\}. 
	\]
	Then there exists a constant $D(\Icc)$ such that
	\[
	\reg(I_n)\le C(\Icc)n+D(\Icc)\quad \text{for all }\ n\gg 0.
	\]
\end{thm}

As illustrated by the following example, the bound provided in the previous theorem is rather sharp. Furthermore, it is predicted that the bound is tight in the case $c=1$; see the discussion below.

\begin{example}
 Let $m$ be a positive integer. Consider the $\Inc$-invariant chain $\Icc=(I_n)_{n\ge 1}$ given by $I_n = \langle \Inc_{1, n} (I_1) \rangle$ if $n \ge 1$ and 
 \[
  I_1=\langle x_{1,1}^m\cdots x_{c,1}^m\rangle.
 \]
Then one has
\[
 I_n=\langle x_{1,1}^m\cdots x_{c,1}^m,\dots,x_{1,n}^m\cdots x_{c,n}^m\rangle\quad\text{for }\ n\ge 1.
\]
Thus, it is evident that
\[
 w_1(\Icc)=\cdots= w_c(\Icc)= \omega(\Icc)=m.
\]
Hence,
\[
 C(\Icc)=\omega(\Icc)-1+\max\Big\{\sum_{k\ne l}w_k(\Icc)\mid l\in[c]\Big\}=cm-1.
\]
Note that
\[
 \reg(I_n)=(cm-1)n+1\quad\text{for all }\ n\ge 1,
\]
as the generators of $I_n$ form a regular sequence; see, e.g., \cite[Theorem 20.2]{Pe11}.
\end{example}

\Cref{thm_bounding_reg_monomial} can be immediately extended to any $\Sym$- and $\Inc$-invariant chain of graded ideals by virtue of \Cref{lem_initial_filtration} (see \cite[Corollaries 4.6, 4.7]{LNNR}).

\begin{cor}
	\label{thm_bounding_reg}
 Let $\Icc=(I_n)_{n\ge 1}$ be a nonzero $\Inc$-invariant chain of graded ideals. Let $\le$ be a monomial order on $R$ respecting $\Inc$. 
 Then there exists a constant $D(\Icc)$ such that
	\[
	\reg(I_n)\le C(\ini_{\le}(\Icc))n+D(\Icc)\quad \text{for all }\ n\gg 0.
	\]
In particular, the conclusion is true if $\Icc$ is a $\Sym$-invariant chain of graded ideals.
\end{cor}

Next, we discuss some special cases of \Cref{conj}. The following result summarizes several instances where \Cref{conj} holds true (see \cite[Propositions 4.14, 4.16, Corollary 6.5]{LNNR}).

\begin{thm}
 Let $\Icc=(I_n)_{n\ge 1}$ be a nonzero $\Inc$-invariant chain of graded ideals. Then \Cref{conj} is true in the following cases:
 \begin{enumerate}
  \item 
  $R_n/I_n$ is an Artinian ring for $n\gg0$.
  
  \item
  There exists an $\Inc$-invariant ideal $I\subseteq R$ such that $I_n=I\cap R_n$ for $n\gg0$, and
  \begin{enumerate}
   \item 
   either $I$ is generated by the $\Inc$-orbit of one monomial, i.e., there is a monomial $u\in R$ such that $I=\langle\Inc(u)\rangle$, or
   
   \item
 $c=1$ and $I$ is a squarefree monomial ideal.
  \end{enumerate}

 \end{enumerate}

\end{thm}

In the remainder of this section we consider the case $c=1$, i.e., $R$ has only one row of variables. For simplicity, we will write the variables of $R$ as $x_i,\ i\in\N$. It is apparent that if $\Icc=(I_n)_{n\ge 1}$ is a nonzero $\Inc$-invariant chain of proper monomial ideals, then $C(\Icc)=\omega(\Icc)-1$. So \Cref{thm_bounding_reg_monomial} yields the following bound:
\[
	\reg(I_n)\le (\omega(\Icc)-1)n+D(\Icc)\quad \text{for }\ n\gg 0,
\]
where $D(\Icc)$ is a suitable constant (see \cite[Corollary 4.8]{LNNR}). Based on computational experiments it is conjectured in \cite[Conjecture 4.12]{LNNR} that this bound is tight:

\begin{conj}
 Let $c=1$ and let $\Icc=(I_n)_{n\ge 1}$ be a nonzero $\Inc$-invariant chain of proper monomial ideals. Then there exists a constant $D(\Icc)$ such that
\[
	\reg(I_n)= (\omega(\Icc)-1)n+D(\Icc)\quad \text{for all }\ n\gg 0.
\]
\end{conj}

This conjecture is a special case of \Cref{conj} with a precise description of the slope of the linear function. Recently, using different approaches Murai \cite{Mu} and Raicu \cite{Ra} have independently verified this conjecture for $\Sym$-invariant chains of monomial ideals. Note that if $\Icc=(I_n)_{n\ge 1}$ is a $\Sym$-invariant chain of monomial ideals, then by \Cref{Noetherianity} there exist monomials $u_1,\dots,u_m\in R$ such that
\[
 I_n=\langle\sigma(u_i)\mid 1\le i\le m,\ \sigma\in \Sym(n)\rangle\quad \text{for }\ n\gg 0.
\]
Evidently, we can choose each $u_i$ of the form $u_i=x_1^{a_{i,1}}x_2^{a_{i,2}}\cdots x_{k_i}^{a_{i,k_i}}$ with
\[
 \ba_i=(a_{i,1},\dots,a_{i,k_i})\in\N^{k_i} \quad\text{and}\quad a_{i,1}\ge a_{i,2}\ge\cdots\ge a_{i,k_i}.
\]
Such an $\ba_i$ is called a \emph{partition of length $k_i$}. Setting $r=\max\{k_i\mid 1\le i\le m\}$ we see that $I_n$ is generated by the $\Sym(n)$-orbits of $u_1,\dots,u_m$ for all $n\ge r$. For brevity, we also say that the chain $\Icc$ is generated by the partitions $\ba_1,\dots,\ba_m$. Notice that $w(u_i)=a_{i,1}$ and
\[
 \omega(\Icc)=\min\{w(u_i)\mid 1\le i\le m\}=\min\{a_{i,1}\mid 1\le i\le m\}.
\]

Using the above notation, the result of Murai \cite[Proposition 3.9]{Mu} and Raicu \cite[Theorem 6.1]{Ra} can be stated as follows:

\begin{thm}
\label{c=1_reg}
  Let $c=1$ and let $\Icc=(I_n)_{n\ge 1}$ be a nonzero $\Sym$-invariant chain of proper monomial ideals. Assume that $\Icc$ is generated by the partitions $\ba_1,\dots,\ba_m$ of length at most $r$. Set $\alpha=(x_1\cdots x_r)^{\omega(\Icc)-1}$. Then one has
  \[
   \reg(I_n)= (\omega(\Icc)-1)n+\reg(I_r:\alpha)\quad  \text{for }\ n\gg 0.
  \]
\end{thm}

It should be noted that Murai and Raicu in fact prove stronger results: Murai \cite[Theorem 1.1]{Mu} obtains the asymptotic behavior of the Betti table of $I_n$ (see \Cref{sec6}), whereas Raicu \cite[Theorem 3.1]{Ra} provides a description of the graded components of the Ext modules $\Ext^j_{R_n}(R_n/I_n,R_n)$ for $j\ge 0$.

\begin{example}
\label{ex411}
 Let $\Icc=(I_n)_{n\ge 1}$ be the $\Sym$-invariant chain of ideals generated by the two partitions $(4,1)$ and $(3,3)$. Then $I_1=\langle0\rangle$ and $I_n$ is generated by the $\Sym(n)$-orbits of $x_1^4x_2,\; x_1^3x_2^3$ for $n\ge 2$. For instance,
 $$
 \begin{aligned}
  I_2&=\langle x_1^4x_2,\; x_1x_2^4,\; x_1^3x_2^3\rangle,\\
  I_3&=\langle x_1^4x_2,\;x_1^4x_3,\;x_1x_2^4,\;x_1x_3^4,\;x_2^4x_3,\;x_2x_3^4,\; x_1^3x_2^3,\;  x_1^3x_3^3,\; x_2^3x_3^3\rangle.
 \end{aligned}
 $$
 We have that $\omega(\Icc)=\min\{4,3\}=3$ and $\alpha=(x_1x_2)^{\omega(\Icc)-1}=x_1^2x_2^2$. Thus,
 \[
  I_2:\alpha=\langle x_1^2,\; x_2^2,\; x_1x_2\rangle
 \]
with regularity $\reg(I_2:\alpha)=2$. So by \Cref{c=1_reg},
\[
 \reg(I_n)= 2n+2\quad  \text{for }\ n\gg 0.
\]
\end{example}

\section{Projective dimension}\label{sec5}

 Analogously to the codimension and Castelnuovo-Mumford regularity, the projective dimension is also expected to grow eventually linearly along $\Inc$-invariant chains of ideals, as proposed in \cite[Conjecture 1.1]{LNNR2}:

\begin{conj}
\label{conj_pd}
 Let $(I_n)_{n\ge 1}$ be an  $\Inc$-invariant chain of ideals. Then $\pd(I_n)$ is eventually a linear function, that is, there exist integers $E$ and $F$ such that
\[
\pd (I_n) =En+F \quad  \text{ whenever }\ n\gg 0.
\]
\end{conj}

This section is devoted to discussing some evidence for this conjecture. First,  recall that one has
\begin{equation}
 \label{eq51}
 \pd (I_n) \ge \codim (I_n)-1,
\end{equation}
with equality if $I_n$ is a Cohen-Macaulay ideal. So \Cref{dim} immediately gives the following (see \cite[Proposition 4.1]{LNNR2}):

\begin{prop}
 \label{CM}
 Let $\Icc=(I_n)_{n\ge 1}$ be an $\Inc$-invariant chain of ideals such that $I_n$ is Cohen-Macaulay for all $n\gg0$. Then \Cref{conj_pd} is true for $\Icc$.
\end{prop}

Note that the ideal $I_n$ generated by the $p$-minors of a $c\times n$ generic matrix is a Cohen-Macaulay ideal; see, e.g., \cite[Theorem 7.3.1]{BH}. So \Cref{CM} is applicable to the chain $\Icc=(I_n)_{n\ge 1}$ considered in \Cref{ex39}. It is also applicable to any chain that is generated by one monomial orbit; see \cite[Corollary 2.2]{GN}.

 Further evidence for \Cref{conj_pd} is given by the next result, which provides upper and lower linear  bounds for $\pd (I_n)$ (see \cite[Propositions 4.3, 4.13]{LNNR2}):

\begin{prop}
 \label{pd both bounds}
 Let $\Icc=(I_n)_{n\ge 1}$ be an $\Inc$-invariant chain of proper ideals. Then there exists an integer $B$ such that
\[
  \gamma (\Icc) n+B \le \pd (I_n) \le cn-1\quad\text{for }\ n\gg 0. 
\]
 In particular, if $c=1$, then there is a positive integer $B'$ such that
 \[
   n-B' \le \pd (I_n) \le n-1\quad\text{for }\ n\gg0. 
\]
\end{prop}

 Note that the upper bound for $\pd (I_n)$ in this proposition follows from Hilbert's Syzygy Theorem, while the lower bound is a consequence of \Cref{dim} and Inequality \eqref{eq51}. In the case that $\Icc=(I_n)_{n\ge 1}$ is a chain of monomial ideals, the lower bound for $\pd (I_n)$ can be considerably improved; see \cite[Theorems 4.6, 4.10]{LNNR2}.

We conclude this section with the following result of Murai \cite[Corollary 3.7]{Mu} which verifies \Cref{conj_pd} for $\Sym$-invariant chains of monomial ideals in the case $c=1$. Together with \Cref{c=1_reg} this is another consequence of Murai's investigation of the asymptotic behavior of Betti tables along such chains, to be discussed in the next section.

\begin{thm}
\label{c=1_pd}
 Let $c=1$ and let $\Icc=(I_n)_{n\ge 1}$ be a nonzero $\Sym$-invariant chain of proper monomial ideals. Assume that $\Icc$ is generated by the partitions $\ba_1,\dots,\ba_m$ of length at most $r$. Then there is a positive integer $B'$ such that
\[
  \pd (I_n) = n-B'\quad\text{for }\ n\ge r. 
\]
\end{thm}

It is worth noting that the integer $B'$ in the previous theorem can be combinatorially determined; see \cite[Theorem on regularity and projective dimension]{Ra}.

\section{Syzygies and Betti tables}\label{sec6}

Given an $\Inc$-invariant chain $\Icc=(I_n)_{n\ge 1}$ of graded ideals, what can be said about the asymptotic behavior of the syzygies of $I_n$? A first answer to this question was provided by Nagel and R\"{o}mer \cite{NR17+}, who established a
stabilization result for the syzygy modules of $I_n$ when $n\gg0$. However, a full description of this interesting result would require the theory of FI- and OI-modules with varying coefficients that is beyond the scope of this survey. Therefore, in this section, we will only discuss a rather informal version of Nagel-R\"{o}mer's result and its consequence on the Betti table of $I_n$. We close the section with a recent result of Murai \cite{Mu}, which provides more information on the Betti table of $I_n$ in the case $c=1$ and $\Icc=(I_n)_{n\ge 1}$ is a $\Sym$-invariant chain of monomial ideals.

Let us begin with a simple example.

\begin{example}
Let $c=1$ and consider the $\Inc$-invariant chain $\Icc=(I_n)_{n\ge 1}$ with
 \[
  I_n=\langle x_jx_k\mid 1\le j<k\le n\rangle\subseteq R_n.
 \]
Then the first syzygies of $I_n$ are given by the following equations:
\begin{align*}
 x_{k+1}(x_jx_k)-x_k(x_jx_{k+1})&=0\quad\text{for }\ j<k,\\
 x_{j+1}(x_jx_k)-x_j(x_{j+1}x_{k})&=0\quad\text{for }\ j+1<k.
\end{align*}
We see that these syzygies are determined through the $\Inc$-action by two first syzygies of $I_3$, namely,
\begin{align*}
 x_{3}(x_1x_2)-x_2(x_1x_{3})&=0,\\
 x_{2}(x_1x_3)-x_1(x_{2}x_{3})&=0.
\end{align*}
\end{example}

This example illustrates a much more general result due to Nagel and R\"{o}mer \cite[Theorem 7.1]{NR17+}. Informally, it implies that for any $\Inc$-invariant chain $\Icc=(I_n)_{n\ge 1}$ of graded ideals and any integer $p\ge0$, the $p$-syzygies of the ideals $I_n$ ``look alike'' eventually. Slightly more precisely, this means that there exists a positive integer $n_p$ such that the $p$-syzygies of $I_n$ are determined through the $\Inc$-action by the $p$-syzygies of $I_{n_p}$ for all $n\ge n_p$. Note that \cite[Theorem 7.1]{NR17+} is actually stated in a more general context of FI- and OI-modules. See also \cite[Theorem A]{Sn} for a related result.

The above result of Nagel and R\"{o}mer leads to the following stabilization of the Betti table (see \cite[Theorem 7.7]{NR17+} for a more general version):

\begin{thm}
\label{betti}
 Let $\Icc=(I_n)_{n\ge 1}$ be an $\Inc$-invariant chain of graded ideals. Then for any integer $p\ge0$ the set
 \[
  \{j\in\Z\mid \beta_{p,j}(I_n)\ne0\}
 \]
stabilizes for $n\gg0$. More precisely, there exist integers $j_0<\cdots<j_t$ depending on $p$ and $\Icc$ such that for $n\gg0$,
\[
 \beta_{p,j}(I_n)\ne0\quad\text{if and only if}\quad j\in\{j_0,\dots,j_t\}.
\]
\end{thm}

This result implies that for any $p\ge 0$, the $p$-th column of the Betti table of $I_n$ has a stable shape whenever $n\gg0$. In the special case when $c=1$ and $\Icc=(I_n)_{n\ge 1}$ is a $\Sym$-invariant chain of monomial ideals, Murai \cite{Mu} recently obtained a much stronger result, yielding the stabilization of the whole Betti table of $I_n$ when $n\gg0$. Before stating his result, let us consider an example.

\begin{example}
 As in \Cref{ex411}, let $\Icc=(I_n)_{n\ge 1}$ be the $\Sym$-invariant chain generated by the two partitions $(4,1)$ and $(3,3)$. The Betti tables of some ideals in this chain, computed by Macaulay2, are shown in \Cref{fig2}.

{\tiny
\begin{figure}[h]
\begin{minipage}[t]{.15\linewidth}
\begin{tabular}[t]{c|cc}
$I_2$&0&1\\
\hline 
\text{5}&\text{2}&\text{.}\\
\text{6}&1&2\\
\end{tabular}
\end{minipage}
\begin{minipage}[t]{.2\linewidth}
\begin{tabular}[t]{c|ccc}
$I_3$&0&1&2\\
\hline 
\text{5}&\text{6}&3&\text{.}\\
\text{6}&3&6&\text{.}\\
\text{7}&\text{.}&\text{.}&\text{.}\\
\text{8}&\text{.}&2&3
\end{tabular}
\end{minipage}
\begin{minipage}[t]{.25\linewidth}
\begin{tabular}[t]{c|cccc}
$I_4$&0&1&2&3\\
\hline 
\text{5}&\text{12}&12&\text{4}&\text{.}\\
\text{6}&6&12&.&\text{.}\\
\text{7}&\text{.}&\text{.}&\text{.}&\text{.}\\
\text{8}&\text{.}&\text{8}&\text{12}&\text{.}\\
\text{9}&\text{.}&\text{.}&\text{.}&\text{.}\\
\text{10}&\text{.}&\text{.}&3&4
\end{tabular}
\end{minipage}
\begin{minipage}[t]{.3\linewidth}
\begin{tabular}[t]{c|ccccc}
$I_5$&0&1&2&3&4\\
\hline 
\text{5}&\text{20}&30&\text{20}&\text{5}&\text{.}\\
\text{6}&10&20&.&.&\text{.}\\
\text{7}&\text{.}&\text{.}&\text{.}&\text{.}&\text{.}\\
\text{8}&\text{.}&20&30&\text{.}&\text{.}\\
\text{9}&\text{.}&\text{.}&\text{.}&\text{.}&\text{.}\\
\text{10}&\text{.}&\text{.}&\text{15}&\text{20}&\text{.}\\
\text{11}&\text{.}&\text{.}&\text{.}&\text{.}&\text{.}\\
\text{12}&\text{.}&\text{.}&\text{.}&\text{4}&\text{5}
\end{tabular}
\end{minipage}\\[10pt]
        
\hfill
\begin{minipage}{.45\linewidth}
\begin{tabular}[t]{c|ccccccc}
$I_6$&0&1&2&3&4&5\\
\hline 
\text{5}&\text{30}&60&\text{60}&\text{30}&\text{6}&\text{.}\\
\text{6}&15&30&.&.&.&\text{.}\\
\text{7}&\text{.}&\text{.}&.&\text{.}&\text{.}&\text{.}\\
\text{8}&\text{.}&\text{40}&\text{60}&.&\text{.}&\text{.}\\
\text{9}&\text{.}&\text{.}&\text{.}&\text{.}&.&\text{.}\\
\text{10}&\text{.}&\text{.}&\text{45}&\text{60}&\text{.}&\text{.}\\
\text{11}&\text{.}&\text{.}&\text{.}&\text{.}&\text{.}&\text{.}\\
\text{12}&.&\text{.}&\text{.}&\text{24}&\text{30}&\text{.}\\
\text{13}&\text{.}&\text{.}&\text{.}&\text{.}&\text{.}&\text{.}\\
\text{14}&.&\text{.}&\text{.}&\text{.}&\text{5}&\text{6}
\end{tabular}
\end{minipage}
\begin{minipage}{.45\linewidth}
\begin{tabular}[t]{c|ccccccc}
$I_6$&0&1&2&3&4&5\\
\hline 
\text{5}&{\color{red}\textcircled{1}}&{\color{red}\textcircled{1}}&{\color{red}\textcircled{1}}&{\color{red}\textcircled{1}}&{\color{red}\textcircled{1}}&\text{.}\\
\text{6}&{\color{green}\textcircled{2}}&{\color{blue}\textcircled{3}}&.&.&.&\text{.}\\
\text{7}&\text{.}&\text{.}&.&\text{.}&\text{.}&\text{.}\\
\text{8}&\text{.}&{\color{green}\textcircled{2}}&{\color{blue}\textcircled{3}}&.&\text{.}&\text{.}\\
\text{9}&\text{.}&\text{.}&\text{.}&\text{.}&.&\text{.}\\
\text{10}&\text{.}&\text{.}&{\color{green}\textcircled{2}}&{\color{blue}\textcircled{3}}&\text{.}&\text{.}\\
\text{11}&\text{.}&\text{.}&\text{.}&\text{.}&\text{.}&\text{.}\\
\text{12}&.&\text{.}&\text{.}&{\color{green}\textcircled{2}}&{\color{blue}\textcircled{3}}&\text{.}\\
\text{13}&\text{.}&\text{.}&\text{.}&\text{.}&\text{.}&\text{.}\\
\text{14}&.&\text{.}&\text{.}&\text{.}&{\color{green}\textcircled{2}}&{\color{blue}\textcircled{3}}
\end{tabular}
\end{minipage}
\caption[]{\begin{varwidth}[t]{.7\linewidth}Betti tables of some $I_n$. The circled numbers in the last table\\ represent 3 line segments.\end{varwidth}}
\label{fig2}
\end{figure}
}
As one might have noticed, these tables suggest that the set
\[
 \{(i,j)\in\Z^2\mid \beta_{i,i+j}(I_n)\ne0\}
\]
of nonzero positions in the Betti table of $I_n$ is a union of \emph{line segments} of length $n-2$. Here, for integers $i,j,s,\ell\ge0$, a line segments of length $\ell$ with starting point $(i,j)$ and slope $s$ is the following subset of $\Z^2$:
\[
 \mathcal{L}((i,j),s,\ell)=\{(i+k,j+sk)\in\Z^2\mid k=0,1,\dots,\ell\}.
\]
In our example, it holds for $n\ge2$ that (see \cite[Proposition 2.7]{Mu}):
\begin{align*}
 \{(i,j)\mid \beta_{i,i+j}(I_n)\ne0\}=\mathcal{L}((0,5),0,n-2)\cup\mathcal{L}((0,6),2,n-2)\cup\mathcal{L}((1,6),2,n-2).
\end{align*}
\end{example}

The main result of Murai \cite[Theorem 1.1]{Mu} asserts that the phenomenon observed in the previous example is true asymptotically for any $\Sym$-invariant chain of monomial ideals:

\begin{thm}
\label{c=1_betti}
 Let $c=1$ and let $\Icc=(I_n)_{n\ge 1}$ be a $\Sym$-invariant chain of monomial ideals. Assume that $\Icc$ is generated by the partitions $\ba_1,\dots,\ba_m$ of length at most $r$. Then there exist finite sets $M\subseteq\Z_{\ge0}^2$ and $N\subseteq\{0,1,\dots,r-1\}$ such that for every $n\ge r$ one has
 \[
  \{(i,j)\mid \beta_{i,i+j}(I_n)\ne0\}=\bigcup_{(i,j)\in M,\; s\in N}\mathcal{L}((i,j),s,n-r)\cup\{(i,j)\mid \beta_{i,i+j}(I_{r-1})\ne0\}.
 \]
\end{thm}

This theorem implies the stabilization of the shape of the Betti table of $I_n$ for $n\ge r$: the nonzero positions in the Betti table of $I_{n+1}$ is obtained from the ones of $I_n$ by just extending each line segment $\mathcal{L}((i,j),s,n-r)$ one point further. Thus, in particular, the shape of the Betti table of $I_r$ completely determines that of $I_n$ for all $n\ge r$.
Interesting consequences of this fact are the asymptotic behavior of $\reg(I_n)$ and $\pd(I_n)$ mentioned earlier (see Theorems \ref{c=1_reg} and \ref{c=1_pd}) 

It should be noted that Murai proved \Cref{c=1_betti} by studying the multigraded components of the Tor modules $\Tor_i(I_n,K)$ through homology modules of simplicial complexes. In fact, he was able to determine the (non-)vanishing of the multigraded Betti numbers of $I_n$ (see \cite[Theorem 3.2]{Mu}).


\section{Open problems}\label{sec7}

The study of asymptotic behavior of $\Inc$-invariant chains is in its early stage and many interesting problems are still open. Some of them are scattered throughout the previous sections. In this section we provide some further problems.

First of all, the following problem arises naturally from \Cref{CM}.

\begin{prob}
\label{pb_CM}
 Characterize those $\Sym$- or $\Inc$-invariant chains of ideals $\Icc=(I_n)_{n\ge 1}$ for which $I_n$ is Cohen-Macaulay whenever $n\gg0$.
\end{prob}

In the case $c=1$ and $\Icc=(I_n)_{n\ge 1}$ is a $\Sym$-invariant chain of monomial ideals, this problem has been resolved recently by Raicu \cite[Theorem on injectivity of maps from Ext to local cohomology]{Ra}:

\begin{thm}
 Let $c=1$ and let $\Icc=(I_n)_{n\ge 1}$ be a $\Sym$-invariant chain of monomial ideals. Assume that $\Icc$ is generated by the partitions $\ba_1,\dots,\ba_m$ of length at most $r$. Write $\ba_i=(a_{i,1},\dots,a_{i,k_i})$. Then the following are equivalent for $n\ge r$:
 \begin{enumerate}
  \item 
  $I_n$ is Cohen-Macaulay.
  
  \item
  $I_n$ is unmixed.
  
  \item
  $a_{i,1}=\cdots=a_{i,p}$ for every $i=1,\dots,m$, where $p=\dim (R_n/I_n)+1$.
 \end{enumerate}
\end{thm}

Under the assumption of the preceding theorem note that $\dim (R_n/I_n)=p-1$ for all $n\ge r$, where $p$ is the minimal length of a partition $\ba_i$; see \cite[Corollary 3.8]{Mu}.

The next problem is closely related to \Cref{pb_CM}.

\begin{prob}
 Let the chain $\Icc=(I_n)_{n\ge 1}$ be $\Sym$- or $\Inc$-invariant. Study the primary decomposition of $I_n$.
\end{prob}

Finally, in view of Theorems \ref{betti} and \ref{c=1_betti} the following problem is of great interest.

\begin{prob}
 Study the asymptotic behavior of Betti tables of ideals of $\Sym$- or $\Inc$-invariant chains.
\end{prob}

It should be mentioned that two questions related to this problem are proposed by Murai in \cite[Questions 5.3, 5.5]{Mu}, where he asks for a generalization of \Cref{c=1_betti} to $\Inc$-invariant chains and for a way to determine the Betti numbers of ideals in $\Sym$-invariant chains.

\noindent {\bf Acknowledgments.}
 We are grateful to the local organizers of the 26th National School on Algebra  (Constanta, Romania, August 26 - September 1, 2018) for their hospitality and to the editors of this volume for their encouragement. We would like to thank the referees for useful comments.

\end{document}